\documentclass[12pt,a4paper]{article}
\usepackage[utf8x]{inputenc}
\usepackage{amsmath}
\usepackage{amsfonts}
\usepackage{amssymb}
\usepackage{amsthm}
\usepackage{mathrsfs}
\usepackage{fancyhdr}
\usepackage{graphicx}
\usepackage[usenames,x11names]{xcolor}
\usepackage{arabtex}
\usepackage{framed}
\usepackage[top=2cm,bottom=2cm,margin=2cm]{geometry}
\usepackage{cancel}
\usepackage{array}   
\usepackage{comment}

\usepackage[affil-it]{authblk}   

\usepackage{hyperref}
\usepackage[numbers,sort&compress]{natbib}  

\title{A polynomial approach to Carlitz's $q$-Bernoulli numbers}

\author[1]{\sc Mohamed Mouzaia}
\author[2]{\sc Bakir FARHI\thanks{Correspending author.}}

\affil[1]{Laboratoire de Math\'ematiques appliqu\'ees,
Facult\'e des Sciences Exactes,
Universit\'e de Bejaia, 06000 Bejaia, Algeria}

\affil[2]{National Higher School of Mathematics,
P.O.Box 75, Mahelma 16093, Sidi Abdellah (Algiers),
Algeria}

\date{}

\let\up=\textsuperscript
\let\epsilon=\varepsilon

\def\R{{\mathbb R}}

\def\N{{\mathbb N}}

\def\idem{\leavevmode\hbox to 10.6mm{\vrule height .63ex depth -.59ex
    width 10mm\hfill}}

\theoremstyle{plain}
\numberwithin{equation}{section}
\newtheorem{thm}{Theorem}[section]

\newtheorem{prop}[thm]{Proposition}
\newtheorem{coll}[thm]{Corollary}

\theoremstyle{definition}

\newtheorem{nota}[thm]{Notation}

\theoremstyle{remark}
\newtheorem{rmk}[thm]{Remark}

\newtheorem{expls}[thm]{Examples}

\newcommand\blfootnote[1]{%
  \begingroup
  \renewcommand\thefootnote{}%
  \footnotetext{#1}%
  \endgroup
}

\begin{document}
\maketitle

\blfootnote{\emph{E-mail addresses}: mohamed.mouzaia@univ-bejaia.dz (M. Mouzaia), bakir.farhi@nhsm.edu.dz (B. Farhi).}

\begin{abstract}
This paper investigates $q$-analogues of the classical Bernoulli polynomials and numbers. We introduce a new polynomial sequence $(B_{n , q}(X))_{n \in \N_0}$, defined via the Jackson integral, and explore its connections with Carlitz's $q$-Bernoulli polynomials and numbers. 
Specifically, we prove that the numbers $B_{n , q}(0)$ are exactly the Carlitz $q$-Bernoulli numbers and that the polynomials $B_{n , q}(X)$ are genuine $q$-analogues of the classical Bernoulli polynomials. This approach leverages the Jackson integral to reformulate Carlitz’s $q$-Bernoulli numbers in terms of classical polynomial structures, offering new insights into their properties. 
\end{abstract}


\section{Introduction and Notation}\label{sec1}

Throughout this paper, we let $\mathbb{N}$ and $\mathbb{N}_0$ denote the sets of positive and nonnegative integers, respectively. For $n \in \mathbb{N}_0$, we let $\mathbb{R}_n[X]$ denote the $\mathbb{R}$-vector subspace of $\mathbb{R}[X]$ consisting of polynomials of degree at most $n$. The Bernoulli polynomials and numbers, denoted $B_n(X)$ and $B_n$ ($n \in \mathbb{N}_0$), are defined via their exponential generating functions:
\[
\frac{t}{e^t - 1} e^{X t} = \sum_{n=0}^{\infty} B_n(X) \frac{t^n}{n!}, \quad \frac{t}{e^t - 1} = \sum_{n=0}^{\infty} B_n \frac{t^n}{n!},
\]
so that $B_n = B_n(0)$ for all $n \in \mathbb{N}_0$. These generating functions yield several key properties of the Bernoulli polynomials and numbers, including:
\begin{align}
B_n(X) &= \sum_{k=0}^n \binom{n}{k} B_k X^{n-k}, \label{eq2} \\
\int_X^{X+1} B_n(t) \, dt &= X^n, \label{eq3}
\end{align}
for all $n \in \mathbb{N}_0$. Notably, the second formula \eqref{eq3} characterizes the $n$-th Bernoulli polynomial: for all $n \in \mathbb{N}_0$, the unique real polynomial $P(X)$ satisfying $\int_X^{X+1} P(t) \, dt = X^n$ is $P(X) = B_n(X)$. As is well known, Bernoulli polynomials and numbers play a central role in number theory, combinatorics, and mathematical analysis, owing to their connections with sums of powers, zeta functions, and the Euler-Maclaurin formula \cite{zag}.

For $q > 0$, $q \neq 1$, the $q$-calculus generalizes classical calculus through the parameter $q$, recovering standard results as $q \to 1$. The $q$-analogue of a number or indeterminate $X$ is defined as $[X]_q := \frac{q^X - 1}{q - 1}$. Fundamental concepts, introduced by Jackson \cite{jac}, are the $q$-derivative and $q$-integral of a function $f$ (satisfying suitable regularity conditions), defined as:
\begin{align*}
D_q f(X) & := \frac{f(q X) - f(X)}{(q - 1) X}, \\
\int_a^b f(t) \, d_q t & := \begin{cases}
(1 - q) b \sum_{n=0}^\infty q^n f(q^n b) - (1 - q) a \sum_{n=0}^\infty q^n f(q^n a), & \text{if } 0 < q < 1, \\
(q - 1) b \sum_{n=0}^\infty q^{-n} f(q^{-n} b) - (q - 1) a \sum_{n=0}^\infty q^{-n} f(q^{-n} a), & \text{if } q > 1,
\end{cases}
\end{align*}
for all $a, b$ in a suitable domain ensuring the well-definition of the expressions. A $q$-antiderivative of a function $f$ is a function $F$ such that $D_q F = f$. The fundamental theorem of $q$-calculus \cite{kac} states that if $f$ is Jackson-integrable on $[a, b] \subset \mathbb{R}$ and $F$ is a $q$-antiderivative of $f$, continuous at $0$, then $\int_a^b f(t) \, d_q t = F(b) - F(a)$.

Next, we define a real $q$-polynomial as any polynomial in $[X]_q$ with real coefficients, equivalently a real polynomial in $q^X$. The degree of a $q$-polynomial is its degree as a polynomial in $[X]_q$ (or $q^X$). The set of all real $q$-polynomials forms an $\mathbb{R}$-vector space with principal bases $([X]_q^k)_{k \in \mathbb{N}_0}$ and $(q^{k X})_{k \in \mathbb{N}_0}$. For $n \in \mathbb{N}_0$, the set of real $q$-polynomials of degree at most $n$ is an $\mathbb{R}$-vector subspace with principal bases $([X]_q^k)_{0 \leq k \leq n}$, $(q^{k X})_{0 \leq k \leq n}$, and $(q^{k X} [X]_q^{n-k})_{0 \leq k \leq n}$, the latter being particularly important for our purposes. Like classical polynomials, two real $q$-polynomials that coincide on an infinite subset of $\mathbb{R}$ are identical (for $q > 0$, $q \neq 1$). To distinguish polynomials (resp. $q$-polynomials) from numbers, we consistently indicate the indeterminate $X$ for polynomials (resp. $q$-polynomials).

In $q$-calculus, some authors investigate $q$-analogues of well-known polynomial sequences, such as Bernoulli, Euler, and Chebyshev polynomials, in the form of $q$-polynomial sequences. For example, in \cite{car}, Carlitz introduced the $q$-polynomial sequences $(\eta_{n,q}(X))_{n \in \mathbb{N}_0}$ and \linebreak $(\beta_{n,q}(X))_{n \in \mathbb{N}_0}$, defined by:
\begin{align}
\eta_{n,q}(X) &= (q - 1)^{-n} \sum_{k=0}^n (-1)^{n-k} \binom{n}{k} \frac{k}{[k]_q} q^{k X}, \label{eq4} \\
\beta_{n,q}(X) &= (q - 1)^{-n} \sum_{k=0}^n (-1)^{n-k} \binom{n}{k} \frac{k + 1}{[k + 1]_q} q^{k X}, \label{eq5}
\end{align}
for all $n \in \mathbb{N}_0$, with the convention $\frac{k}{[k]_q} = 1$ for $k = 0$. Carlitz showed that the second $q$-polynomial sequence $(\beta_{n , q}(X))_{n \in \N_0}$ converges to the classical Bernoulli polynomial sequence as $q \to 1$, i.e., $\lim_{q \to 1} \beta_{n,q}(X) = B_n(X)$ for all $n \in \mathbb{N}_0$. He also considered the values of $\eta_{n , q}(X)$ and $\beta_{n , q}(X)$ at $X = 0$, respectively denoted by $\eta_{n , q}$ and $\beta_{n , q}$, and proved the important formulas:
\begin{align}
\eta_{n,q}(X) &= \sum_{k=0}^n \binom{n}{k} \eta_{k,q} q^{k X} [X]_q^{n-k}, \label{eq6} \\
\beta_{n,q}(X) &= \sum_{k=0}^n \binom{n}{k} \beta_{k,q} q^{k X} [X]_q^{n-k}, \label{eq7}
\end{align}
for all $n \in \mathbb{N}_0$. Thus, $\beta_{n,q}$ are $q$-analogues of the Bernoulli numbers $B_n$, and Formula \eqref{eq7} is a $q$-analogue of Formula \eqref{eq2}. For instance, the first four Carlitz $q$-Bernoulli numbers are
\begin{equation}\label{eq8}
\beta_{0 , q} = 1 ~,~ \beta_{1 , q} = - \dfrac{1}{[2]_q} ~,~ \beta_{2 , q} = \dfrac{q}{[2]_q [3]_q} ~,~ \text{ and } \beta_{3 , q} = - \dfrac{q (q - 1)}{[3]_q [4]_q} . 
\end{equation}
For further reading on Carlitz's $q$-Bernoulli numbers and polynomials and their extensions, see, e.g., \cite{far}.

This paper introduces, via the Jackson integral, a new sequence of polynomials, \linebreak $(B_{n,q}(X))_{n \in \N_0}$, and studies its relationship with Carlitz’s $q$-polynomial sequences $(\eta_{n,q}(X))_{n \in \mathbb{N}_0}$ and $(\beta_{n,q}(X))_{n \in \mathbb{N}_0}$. Our objectives are twofold: first, to establish that $B_{n,q}(0) = \beta_{n,q}$ for all $n \in \mathbb{N}_0$; second, to prove that $B_{n,q}(X)$ are genuine $q$-analogues of the Bernoulli polynomials, that is, $\lim_{q \to 1} B_{n,q}(X) = B_n(X)$ for all $n \in \N_0$. This approach provides a novel perspective on Carlitz’s $q$-Bernoulli numbers $\beta_{n,q}$ by leveraging the Jackson integral to reformulate their associated $q$-polynomial sequences in terms of classical polynomial structures.

\section{The results and the proofs}

Our \textit{$q$-Bernoulli polynomials} $B_{n , q}(X)$ are defined through the following proposition.

\begin{prop}\label{p1}
For all $n \in \N_0$, there exists a unique real polynomial $B_{n , q}(X)$ satisfying:
\begin{equation}\label{eq1}
\int_{X}^{q X + 1} B_{n , q}(t) \, d_q t = (q - 1) X^{n + 1} + X^n .
\end{equation}
\end{prop}

\begin{proof}
Fix $n \in \N_0$. Comparing the degrees of the polynomials on both sides of \eqref{eq1}, we see that if $B_{n , q}(X)$ exists, it must be of degree $n$; so $B_{n , q}(X) \in \R_n[X]$. To establish the existence and the uniqueness of $B_{n , q}(X)$ in $\R_n[X]$, consider the $\R$-vector subspace $E_n$ of $\R_{n + 1}[X]$ defined by:
$$
E_n := \left\{P \in \R_{n + 1}[X] :~ P\left(\frac{1}{1 - q}\right) = 0\right\}
$$
and the linear mapping $\varphi_n$ defined by:
$$
\begin{array}{rcl}
\varphi_n :~ \R_n[X] & \longrightarrow & E_n \\
P & \longmapsto & \int_{X}^{q X + 1} P(t) \, d_q t .
\end{array}
$$
We claim that $\varphi_n$ is an isomorphism of $\R$-vector spaces. First, we show that $\varphi_n$ is injective. Suppose $P \in \ker{\varphi_n}$, so that
$$
\varphi_n(P) = \int_{X}^{q X + 1} P(t) \, d_q t = 0_{\R[X]} .
$$
Evaluating at $X = [k]_q$ ($k \in \N_0$), we have $q [k]_q + 1 = [k + 1]_q$, so
$$
\int_{[k]_q}^{[k + 1]_q} P(t) \, d_q t = 0 .
$$
Summing these integrals from $k = 0$ to $k = N - 1$ for $N \in \N$, we obtain
$$
\int_{0}^{[N]_q} P(t) \, d_q t = 0 \quad (\forall N \in \N) .
$$
Let $Q$ be the polynomial $q$-antiderivative of $P$ such that $Q(0) = 0$, i.e., $Q(X) = \int_{0}^{X} P(t) \, d_q t$. Then, the above implies $Q([N]_q) = 0$ for all $N \in \N$. Since the points $[N]_q$ ($N \in \N$) are pairwise distinct (because $q > 0$ and $q \neq 1$), it follows that $Q = 0_{\R[X]}$. Thus $P = D_q Q = 0_{\R[X]}$. Hence $\ker{\varphi_n} = \{0_{\R_n[X]}\}$, and $\varphi_n$ is injective. Furthermore, since $\dim{\R_n[X]} = \dim{E_n} = n + 1$, then $\varphi_n$ is an isomorphism, as we claimed it to be. Consequently, the polynomial $(q - 1) X^{n + 1} + X^n = X^n \left(1 - (1 - q) X\right)$ (which clearly belongs to $E_n$) has a unique preimage by $\varphi_n$. This ensures the existence and the uniqueness of $B_{n , q}(X)$ in $\R_n[X]$ and completes the proof. 
\end{proof}

\begin{rmk}
Assuming the existence of the limit $\lim_{q \to 1} B_{n , q}(t)$, which we provisionally denote by $\widetilde{B}_n(t)$, for $n \in \N_0$, by setting $q \to 1$ in Formula \eqref{eq1}, we obtain   
$$
\int_{X}^{X + 1} \widetilde{B}_n(t) \, d t = X^n ,
$$
which is one of the characterizing properties of the $n$\up{th} Bernoulli polynomial $B_n(t)$. Thus, $\lim_{q \to 1} B_{n , q}(t) = B_n(t)$, justifying the terminology ``$q$-Bernoulli polynomials'' for the $B_{n , q}(t)$'s. However, establishing the existence of $\lim_{q \to 1} B_{n , q}(t)$ for $n \in \N_0$ is far from trivial. At the end of the paper, we will rigorously show (by another technique) that indeed $\lim_{q \to 1} B_{n , q}(t) = B_n(t)$ for all $n \in \N_0$ (see Theorem \ref{t2}).
\end{rmk}

\begin{nota}
For all $n \in \N_0$, we denote by $F_{n , q}(X)$ the polynomial given by
$$
F_{n , q}(X) := \int_{0}^{X} B_{n , q}(t) \, d_q t ,
$$
which is a $q$-antiderivative of $B_{n , q}(X)$ satisfying $F_{n , q}(0) = 0$.
\end{nota}

\begin{expls}
To compute the polynomials $B_{n , q}(X)$ explicitly for $n \in \N_0$, we express $B_{n , q}(X) = \sum_{k = 0}^{n} a_{n , q}^{(k)} X^k$, where $a_{n , q}^{(k)} \in \R$ for $k = 0 , 1 , \dots , n$. Substituting into \eqref{eq1} and equating coefficients of the polynomials on both sides yields a Cramer system of $(n + 1)$ equations in the $(n + 1)$ unknowns $a_{n , q}^{(k)}$ ($k = 0 , 1 , \dots , n$). Solving this system gives the polynomial $B_{n , q}(X)$. For example, we obtain
\begin{align*}
B_{0 , q}(X) & = 1 , \\
B_{1 , q}(X) & = - \frac{1}{q + 1} + X , \\
B_{2 , q}(X) & = \frac{q}{(q + 1) (q^2 + q + 1)} - \frac{2 q + 1}{q^2 + q + 1} X + X^2 .
\end{align*}
\end{expls}

Our $q$-Bernoulli numbers are naturally defined as $B_{n , q} := B_{n , q}(0)$ for $n \in \N_0$. Thus, we have
$$
B_{0 , q} = 1 , \quad B_{1 , q} = - \frac{1}{[2]_q} , \quad B_{2 , q} = \frac{q}{[2]_q [3]_q} .
$$
For $n = 0 , 1 , 2$, we observe that the numbers $B_{n , q}$ coincide with the Carlitz $q$-Bernoulli numbers $\beta_{n , q}$ (see \eqref{eq8}). In the following, we will prove that this holds for all $n \in \N_0$ and establish a general connection between the polynomials $B_{n , q}(X)$ and the Carlitz $q$-Bernoulli polynomials $\beta_{n , q}(X)$. To this end, we first introduce some preliminary results. The following proposition is taken from \cite{far}, and for the sake of completeness, we include its proof here.

\begin{prop}[{\cite[Corollary 2.13]{far}}]\label{p2}
For all positive integers $n$ and $N$, we have
\[
\sum_{k = 0}^{N - 1} q^k [k]_q^{n - 1} = \dfrac{\eta_{n , q}(N) - \eta_{n , q}}{n} .
\]
\end{prop}

\begin{proof}
Let $n , N \in \N$ be fixed. We have
\begin{align*}
\sum_{k = 0}^{N - 1} q^{k} {[k]_q}^{n - 1} & = \sum_{k = 0}^{N - 1} q^k \left(\frac{q^k - 1}{q - 1}\right)^{n - 1} \\
& = (q - 1)^{- n + 1} \sum_{k = 0}^{N - 1} q^k \left(q^k - 1\right)^{n - 1} \\
& = (q - 1)^{- n + 1} \sum_{k = 0}^{N - 1} q^k \sum_{i = 0}^{n - 1} \binom{n - 1}{i} q^{k i} (-1)^{n - 1 - i} \quad (\text{by the binomial formula}) \\
& = (q - 1)^{- n + 1} \sum_{i = 0}^{n - 1} (-1)^{n - 1 - i} \binom{n - 1}{i} \sum_{k = 0}^{N - 1} q^{(i + 1) k} \\
& = (q - 1)^{- n + 1} \sum_{i = 0}^{n - 1} (-1)^{n - 1 - i} \binom{n - 1}{i} \frac{q^{(i + 1)N} - 1}{q^{i + 1} - 1} \\
& = (q - 1)^{- n + 1} \sum_{i = 1}^{n} (-1)^{n - i} \binom{n - 1}{i - 1} \frac{q^{i N} - 1}{q^i - 1} .
\end{align*}
Using the identities $\binom{n - 1}{i - 1} = \frac{i}{n} \binom{n}{i}$ and $q^i - 1 = (q - 1) [i]_q$ for $1 \leq i \leq n$, it follows that:
\begin{align*}
\sum_{k = 0}^{N - 1} q^{k} {[k]_q}^{n - 1} & = \frac{1}{n} (q - 1)^{- n} \sum_{i = 1}^{n} (-1)^{n - i} \binom{n}{i} \frac{i}{[i]_q} \left(q^{i N} - 1\right) \\
& = \frac{1}{n} \left(\eta_{n , q}(N) - \eta_{n , q}(0)\right) \quad (\text{by \eqref{eq4}}) ,
\end{align*}
as required.
\end{proof}

We also require the following results concerning the derivatives of certain $q$-polynomials.

\begin{prop}\label{p3}
For all positive integer $n$, we have
\[
\eta_{n , q}'(X) = n \frac{\log q}{q - 1} q^X \beta_{n - 1 , q}(X) .
\]
\end{prop}

\begin{proof}
From \eqref{eq4}, we have for all $n \in \N$:
\begin{align*}
\eta_{n , q}'(X) & = (q - 1)^{- n} \sum_{k = 0}^{n} (-1)^{n - k} \binom{n}{k} \frac{k}{[k]_q} \left(q^{k X}\right)' \\
& = (q - 1)^{- n} \sum_{k = 1}^{n} (-1)^{n - k} \binom{n}{k} \frac{k}{[k]_q} (k \log q) q^{k X} .
\end{align*}
Using the identity $k \binom{n}{k} = n \binom{n - 1}{k - 1}$ for $1 \leq k \leq n$, we rewrite the sum as
\begin{align*}
\eta_{n , q}'(X) & = n \frac{\log q}{q - 1} \cdot (q - 1)^{- n + 1} \sum_{k = 1}^{n} (-1)^{n - k} \binom{n - 1}{k - 1} \frac{k}{[k]_q} q^{k X} \\
& = n \frac{\log q}{q - 1} \cdot (q - 1)^{- n + 1} \sum_{k = 0}^{n - 1} (-1)^{n - k - 1} \binom{n - 1}{k} \frac{k + 1}{[k + 1]_q} q^{(k + 1) X} \\
& = n \frac{\log q}{q - 1} q^X \beta_{n - 1 , q}(X) \quad (\text{according to \eqref{eq5}}) .
\end{align*}
The proposition is proved.
\end{proof}

\begin{prop}\label{p4}
For all $n \in \N_0$, we have
\[
B_{n , q}(0) = F_{n , q}'(0) .
\]
\end{prop}

\begin{proof}
Given $n \in \N_0$, we have
\[
B_{n , q}(0) = \lim_{X \to 0} B_{n , q}(X) = \lim_{X \to 0} D_q F_{n , q}(X) = \lim_{X \to 0} \dfrac{F_{n , q}(q X) - F_{n , q}(X)}{(q - 1) X} .
\]
This limit is of indeterminate form $\frac{0}{0}$. Applying l'Hôpital's rule, we get
\[
B_{n , q}(0) = \lim_{X \to 0} \dfrac{q F_{n , q}'(q X) - F_{n , q}'(X)}{q - 1} = F_{n , q}'(0) ,
\]
as required.
\end{proof}

We are now ready to state and prove our main results connecting the polynomials $B_{n , q}(X)$ and their $q$-antiderivatives $F_{n , q}(X)$ with the Carlitz $q$-polynomials $\eta_{n , q}(X)$ and $\beta_{n , q}(X)$. We begin with the following key theorem.

\begin{thm}\label{t1}
For all $n \in \N_0$, we have
\[
F_{n , q}\left([X]_q\right) = \dfrac{\eta_{n + 1 , q}(X) - \eta_{n + 1 , q}}{n + 1} .
\]
\end{thm}

\begin{proof}
Fix $n \in \N_0$. For $k \in \N_0$, evaluate \eqref{eq1} at $X = [k]_q$. Since $q [k]_q + 1 = [k + 1]_q$ and $(q - 1) [k]_q^{n + 1} + [k]_q^n = q^k [k]_q^n$, we obtain
\[
\int_{[k]_q}^{[k + 1]_q} B_{n , q}(t) \, d_q t = q^k [k]_q^n .
\]
Summing from $k = 0$ to $k = N - 1$ for $N \in \N$, we get
\[
F_{n , q}([N]_q) = \int_{0}^{[N]_q} B_{n , q}(t) \, d_q t = \sum_{k = 0}^{N - 1} q^k [k]_q^n .
\]
By Proposition \ref{p2}, this equals
\[
F_{n , q}\left([N]_q\right) = \dfrac{\eta_{n + 1 , q}(N) - \eta_{n + 1 , q}}{n + 1}
\]
for all $N \in \N$. This shows that the $q$-polynomials $F_{n , q}([X]_q)$ and $\frac{\eta_{n + 1 , q}(X) - \eta_{n + 1 , q}}{n + 1}$ take the same values on $\N$, implying that they are identical; that is,
\[
F_{n , q}\left([X]_q\right) = \dfrac{\eta_{n + 1 , q}(X) - \eta_{n + 1 , q}}{n + 1} ,
\]
as required.
\end{proof}

\begin{coll}\label{c1}
For all $n \in \N_0$, we have
\[
F_{n , q}'\left([X]_q\right) = \beta_{n , q}(X) .
\]
\end{coll}

\begin{proof}
Fix $n \in \N_0$. Differentiating the formula in Theorem \ref{t1} and applying Proposition \ref{p3}, we obtain
\[
\frac{\log q}{q - 1} q^X F_{n , q}'\left([X]_q\right) = \frac{\log q}{q - 1} q^X \beta_{n , q}(X) ,
\]
implying the required result.
\end{proof}

\begin{coll}\label{c2}
For all $n \in \N_0$, we have
\[
B_{n , q} = \beta_{n , q} .
\]
On the other words, our $q$-Bernoulli numbers coincide with the Carlitz $q$-Bernoulli numbers.
\end{coll}

\begin{proof}
Take $X = 0$ in the formula of Corollary \ref{c1} and use Proposition \ref{p4}.
\end{proof}

We conclude this paper by proving that the polynomials $B_{n , q}(X)$ ($n \in \N_0$) are genuine $q$-analogues of the classical Bernoulli polynomials; that is, $B_{n , q}(X)$ tends to $B_n(X)$ as $q \to 1$.

\begin{thm}\label{t2}
For all $n \in \N_0$, we have
\[
\lim_{q \to 1} B_{n , q}(X) = B_n(X) .
\]
\end{thm}

\begin{proof}
Fix $n \in \N_0$. By \eqref{eq7}, we have
\begin{align*}
\beta_{n , q}(t) & = \sum_{k = 0}^{n} \binom{n}{k} \beta_{k , q} q^{k t} [t]_q^{n - k} \\
& = \sum_{k = 0}^{n} \binom{n}{k} \beta_{k , q} \left((q - 1) [t]_q + 1\right)^k [t]_q^{n - k} .
\end{align*}
By identifying this with the formula in Corollary \ref{c1}, we derive that:
\begin{align*}
F_{n , q}'(t) & = \sum_{k = 0}^{n} \binom{n}{k} \beta_{k , q} \left((q - 1) t + 1\right)^k t^{n - k} \\
& = \sum_{k = 0}^{n} \binom{n}{k} \beta_{k , q} \left(\sum_{i = 0}^{k} \binom{k}{i} (q - 1)^i t^i\right) t^{n - k} \quad (\text{by the binomial formula}) \\
& = \sum_{k = 0}^{n} \sum_{i = 0}^{k} \binom{n}{k} \binom{k}{i} \beta_{k , q} (q - 1)^i t^{n + i - k} .
\end{align*}
Then, integrating from $0$ to $X$ yields (since $F_{n , q}(0) = 0$):
\[
F_{n , q}(X) = \sum_{k = 0}^{n} \sum_{i = 0}^{k} \binom{n}{k} \binom{k}{i} \beta_{k , q} (q - 1)^i \dfrac{X^{n + i - k + 1}}{n + i - k + 1} .
\]
Finally, Applying the $q$-derivative $D_q$ to both sides, we obtain
\[
B_{n , q}(X) = \sum_{k = 0}^{n} \sum_{i = 0}^{k} \binom{n}{k} \binom{k}{i} \beta_{k , q} (q - 1)^i \dfrac{[n + i - k + 1]_q}{n + i - k + 1} X^{n + i - k} .
\]
As $q \to 1$, we have $\beta_{k , q} \to B_k$ and $[n + i - k + 1]_q \to n + i - k + 1$ (for all integers $k$ and $i$ such that $0 \leq i \leq k \leq n$); hence
\[
\lim_{q \to 1} B_{n , q}(X) = \sum_{k = 0}^{n} \binom{n}{k} B_k X^{n - k} = B_n(X)
\]
(according to \eqref{eq2}). This completes the proof.
\end{proof}

\begin{rmk}
From Corollary \ref{c1}, the connection between our polynomials $B_{n , q}(X)$ and the Carlitz $q$-Bernoulli polynomials $\beta_{n , q}(X)$ is given by
\[
\beta_{n , q}(X) = \left(\int_{0}^{s} B_{n , q}(t) \, d_q t\right)' \bigg\vert_{s = [X]_q}
\]
for all $n \in \N_0$. A similar formula for $\beta_{n , q}(X)$ ($n \in \N_0$), derived directly from its definition \eqref{eq5}, is given by:
\[
\beta_{n , q}(X) = \left(\int_{0}^{s} \left(\dfrac{t - 1}{q - 1}\right)^n \, d_q t\right)' \bigg\vert_{s = q^X} \quad (\forall n \in \N_0) .
\]
\end{rmk}

\rhead


{\it References}

\begin{thebibliography}{99}
\bibitem{zag}
{\sc T. Arakawa, T. Ibukiyama, M. Kaneko, and D. Zagier}. {\it Bernoulli numbers and zeta functions}, Springer Monographs in Mathematics, Springer, New York, NY, 2014. 

\bibitem{car}
{\sc L. Carlitz}. $q$-Bernoulli numbers and polynomials, {\it Duke Math. J.}, {\bf 15}, n°4 (1948), p. 987-1000.

\bibitem{far}
{\sc B. Farhi}. $q$-analogues of sums of consecutive powers of natural numbers and extended Carlitz q-Bernoulli numbers and polynomials, preprint, 2024. Available at \href{https://arxiv.org/pdf/2501.02499v1}{https://arxiv.org/pdf/2501.02499v1}.

\bibitem{jac}
F. H. Jackson. On $q$-functions and a certain difference operator, {\it Earth and Environmental Science Transactions of the Royal Society of Edinburgh}, {\bf 46} n°2 (1909), p. 253-281.

\bibitem{kac}
{\sc V. Kac and P. Cheung}. {\it Quantum Calculus}, Universitext, Springer-Verlag, 2002.

\end{thebibliography}
\end{document}